\documentclass[11pt]{article}

\usepackage{amsmath}
\usepackage{amssymb}
\usepackage{bbm}

\newtheorem{theorem}{Theorem}

\newtheorem{corollary}[theorem]{Corollary}

 \newcommand{\beq}{\begin{equation}}
\newcommand{\eeq}{\end{equation}}

\newenvironment{proof}{\medbreak\noindent{\it Proof.}\rm}{\hfill$\square$\rm}

\newcommand{\Rn}{{ \mathbb R}^n}
\newcommand{\C}{{\mathbb  C}}
\newcommand{\D}{{\mathbb  D}}
\newcommand{\Cn}{{\mathbb  C}^n}
\newcommand{\Rnm}{{\mathbb R}_-^n}
\newcommand{\Rnp}{{\mathbb R}_+^n}

\newcommand{\cN}{{\mathcal N}}
\newcommand{\cP}{{\mathcal P}}

\newcommand{\E}{{\mathcal E}}
\newcommand{\cL}{{\mathcal  L}}

\newcommand{\Capa}{{\operatorname{Cap}\,}}
\newcommand{\Log}{{\operatorname{Log}\,}}
\newcommand{\Exp}{{\operatorname{Exp}\,}}
\newcommand{\Vol}{{\operatorname{Vol}}}
\newcommand{\Covol}{{\operatorname{Covol}}}

\begin{document}

\begin{center}{\Large\bf Interpolation of weighted extremal functions}
\end{center}

\begin{center}{\large Alexander Rashkovskii}
\end{center}

\bigskip

\begin{abstract} An approach to interpolation of compact subsets of $\Cn$, including Brunn-Minkowski type inequalities for the capacities of the interpolating sets, was developed in \cite{CER} by means of plurisubharmonic geodesics between relative extremal functions of the given sets. Here we show that a much better control can be achieved by means of the geodesics between {\sl weighted} relative extremal functions. In particular, we establish convexity properties of the capacities that are stronger than those given by the Brunn-Minkowski inequalities.

\medskip\noindent
{\sl Mathematic Subject Classification}: 32U15, 32U20, 52A20, 52A40
\end{abstract}

\section{Introduction}

Classical complex interpolation of Banach spaces, due to Calder\'{o}n \cite{Ca} (see \cite{BL} and, for more recent developments, \cite{CEK}) is based on constructing holomorphic hulls generated by certain families of holomorphic mappings. A slightly different approach proposed in \cite{CER} rests on plurisubharmonic geodesics. The notion has been originally considered, starting from 1987, for metrics on compact K\"{a}hler manifolds (see \cite{G12} and the bibliography therein), while its local counterpart for plurisubharmonic functions on bounded hyperconvex domains of $\Cn$ was introduced more recently in \cite{BB} and \cite{R17a}, see also \cite{A}.

In the simplest case, the geodesics we need can be described as follows. Denote by
$A=\{\zeta\in\C:\:0< \log|\zeta| < 1\}$
the annulus bounded by the circles
$A_j=\{\zeta:\: \log|\zeta|=j\}$, $j=0,1$. Let $\Omega$ be a bounded hyperconvex domain in $\Cn$.
Given two plurisubharmonic functions $u_0,u_1$ in $\Omega$, equal to zero on $\partial \Omega$, we
consider the class $W$
of all plurisubharmonic functions $u(z,\zeta)$ in $\Omega\times A$, such that $$\limsup_{\zeta\to A_j} u(z,\zeta)\le u_j(z)\quad \forall z\in \Omega.$$
Its Perron envelope $\cP_W(z,\zeta)=\sup\{u(z,\zeta):\: u\in W\}$
belongs to the class and satisfies $\cP_W(z,\zeta)=\cP_W(z,|\zeta|)$, which gives rise to the functions $$u_t(z):=\cP_W(z,e^t), \quad 0<t<1,$$ the {\it geodesic} between $u_0$ and $u_1$. When the functions $u_j$ are bounded, the geodesic $u_t$ tends to $u_j$ as $t\to j$, uniformly on $\Omega$. One of the main properties of the geodesics is that they linearize the pluripotential energy functional
$$  \E(u)=\int_\Omega u(dd^c u)^n, $$
which means
\beq\label{lin} \E(u_t)=(1-t)\,\E(u_0)+t\,\E(u_1),\eeq see the details in \cite{BB} and \cite{R17a}.

In \cite{R17a}, this was adapted to the case when the endpoints $u_j$ are {\it relative extremal functions} $\omega_{K_j}$ of non-pluripolar compact sets $K_0,K_1\subset\Omega$; we recall that
$$
\omega_K(z)=\omega_{K,\Omega}(z)=\limsup_{y\to z}\, \cP_{\cN_K}(y),
$$
where ${\cN_K}$ is the collection of all negative plurisubharmonic functions $u$ in $\Omega$ with $u|_K\le -1$, see \cite{Kl}.
Note that
$$ \E(\omega_{K})=-\Capa(K),$$
where
$$ \Capa(K)=\Capa(K, \Omega)=(dd^c\omega_K)^n(\Omega)=(dd^c\omega_K)^n(K)$$
is the {\it  Monge-Amp\`ere capacity} of $K$ relative to $\Omega$.

If each $K_j$ is polynomially convex (i.e., coincides with its polynomial hull), then the functions $u_j=-1$ exactly on $K_j$, are continuous on $\overline\Omega$, and the geodesics $u_t\in C(\overline\Omega\times [0,1])$. Let
\beq\label{K_t} K_t=\{z\in \Omega:\: u_t(z)=-1\},\quad  0< t< 1,\eeq
then (\ref{lin}) implies
\beq\label{lincap} \Capa(K_t)\le (1-t)\,\Capa(K_0)+t\,\Capa(K_1).\eeq

As was shown in \cite{R17b}, the functions $u_t$ in general are different from the relative extremal functions of $K_t$. Moreover, if the sets $K_j$ are Reinhardt (toric), then $u_t =\omega_{K_t}$ for some $t\in (0,1)$ only if $K_0=K_1$, so an equality in (\ref{lincap}) is never possible unless the geodesic degenerates to a point.

Furthermore, in the toric case the capacities (with respect to the unit polydisk $\D^n$) were proved in \cite{CER} to be not just convex functions of $t$, as is depicted in ({\ref{lincap}), but logarithmically convex:
\beq\label{logcap} \Capa(K_t,\D^n)\le \Capa(K_0,\D^n)^{1-t}\,\Capa(K_1,\D^n)^t.\eeq
This was done by representing the capacities, due to \cite{ARZ}, as (co)volumes of certain sets in $\Rn$ and applying convex geometry methods to an operation of {\it copolar addition} introduced in \cite{R17b}. Furthermore, the sets $K_t$ in the toric situation were shown to be the geometric means (multiplicative combinations) of $K_j$,
exactly as in the Calder\'{o}n complex interpolation theory. And again, an equality in (\ref{logcap}) is possible only if $K_0=K_1$. It is worth mentioning that the {\sl  volumes} of $K_t$ satisfy the opposite Brunn-Minkowski inequality \cite{CE}:
 $$\Vol(K_t)\ge \Vol(K_0)^{1-t} \Vol(K_1)^t.$$

In this note, we apply the geodesic technique to {\it weighted relative extremal functions} $$u_j^c=c_j\,\omega_{K_j},\quad c_j>0,$$ the sets $K_t$ being replaced with the sets
$ K^c_t$ where the functions $u_t^c$ attain their minimal values, $-c_t$. The function $t\mapsto c_t$ turns out to be convex; moreover, it is actually linear, $c_t=(1-t)\,c_0+t\,c_1$, provided $K_0\cap K_1\neq \emptyset$.
With such an interpolation, one can have $u_t^c=c_t\,\omega_{K^c_t,\Omega}$ for a non-degenerate geodesic, in which case there is no loss in the transition from the functional $\E(u_t^c)$ to the capacity $\Capa(K^c_t)$. And in any case, we establish the weighted inequality
$$ c_t^{n+1}\Capa(K^c_t)\le (1-t)\,c_0^{n+1}\Capa(K_0)+t\,c_1^{n+1}\Capa(K_1)$$
which, for a smart choice of the constants $c_j$, is stronger than (\ref{lincap}) and even (in the toric case) than (\ref{logcap}). In particular, it implies that the function
$$t\mapsto \left(\Capa(K_t^c)\right)^{-\frac1{n+1}}$$
is concave.

In the toric setting of Reinhardt sets $K_j$ in the unit polydisk, we show that the interpolating sets $K_t^c$ actually are the geometric means, so they
do not depend on the weights $c_j$ and coincide with the sets $K_t$ in the non-weighted interpolation; we don't know if the latter is true in the general, non-toric setting.

Finally, we transfer the above results on the capacities of sets in $\Cn$ to the realm of convex geometry, developing thus the Brunn-Minkowski theory for volumes of (co)convex sets in $\Rn$ \cite{KT}, \cite{R17b}, \cite{CER}.

\section{General setting}
Here we consider the general case of $u_j^c=c_j\,\omega_{K_j}$ with $c_j>0$ and $K_j$ non-pluripolar, compact, polynomially convex subsets of a bounded hyperconvex domain $\Omega$ of $\Cn$. In this situation, the functions $u_j^c(z)=-c_j$ precisely on $K_j$ and are continuous on $\overline\Omega$, the geodesics $u_t$ converge to $u_j$, uniformly on $\Omega$, as $t\to j$, and belong to $C(\overline\Omega\times [0,1])$, as in the non-weighted case dealt with in \cite{R17a} and \cite{CER}.

Denote
$$c_t=-\min\{u_t^c(z):\: z\in \Omega\}$$
and
\begin{equation}\label{K_t^c} K^c_t=\{z\in \Omega:\: u^c_t(z)=-c_t\},\quad  0< t< 1, \end{equation}
the set where $u_t^c$ attains its minimal value on $\Omega$.

\begin{theorem}\label{gentheo} In the above setting, we have:
\begin{enumerate}
\item[(i)] $ c_t\le (1-t)\,c_0+t\,c_1$, with an equality if $K_0\cap K_1\neq \emptyset$;
\item[(ii)] the function $t\mapsto c_t^{n+1}\Capa(K_t)$ is convex:
\beq\label{wlincap} c_t^{n+1}\Capa(K^c_t)\le (1-t)\,c_0^{n+1}\Capa(K_0)+t\,c_1^{n+1}\Capa(K_1);\eeq
\item[(iii)] if the weights $c_j$ are chosen such that
\beq\label{equil} c_0^{n+1}\Capa(K_0)=c_1^{n+1}\Capa(K_1),\eeq
then the function
$$ V(t):=\left(\Capa(K_t^c)\right)^{-\frac1{n+1}}$$
is concave and, consequently, the function
$$ \rho(t):=V(t)^{-1}=\left(\Capa(K_t^c)\right)^{\frac1{n+1}}$$
is convex.
\end{enumerate}
\end{theorem}

\begin{proof}
{\it (i)}
Consider $v_j = c_j\,\omega_K$ for $j=0,1$,
where $K=K_0\cup K_1$. The set $K$ might be not polynomially convex, but $\omega_K$ is nevertheless a bounded plurisubharmonic function on $\Omega$ with zero boundary values, so the geodesic $v_t^c$ is well defined and converge to $v_j$, uniformly on $\Omega$, as $t\to j$ \cite[Prop.~3.1]{R17a}.
Since $v_j\le u_j^c$, we have $v_t^c\le u_t^c$. Assume $c_0\ge c_1$, then the corresponding geodesic $v_t^c = \max\{c_0\,\omega_K, -((1-t)\,c_0+t\,c_1)\}$ because the right-hand side is maximal in $\Omega\times A$ and has the prescribed boundary values at $t=0$ and $t=1$. Therefore,
$$-c_t\ge \min\{v_t^c(z):\: z\in \Omega\}\ge -((1-t)\,c_0+t\,c_1),$$
which proves the convexity of $c_t$.

To finish the proof of {\it (i)}, let $z^*\in K_0\cap K_1\neq\emptyset$, then $-c_t\le u_t^c(z^*)$. Since the convexity of the function $u_t^c(z^*)$ in $t$ implies $u_t^c(z^*)\le -((1-t)\,c_0+t\,c_1)$, we get $ c_t\ge (1-t)\,c_0+t\,c_1$ and thus the linearity.

{\it (ii)} Since $u_j^c=c_j\,\omega_{K_j}$, we have
$$\E(u_j)= c_j^{n+1}\int_\Omega (dd^c \omega_{K_j})^n = -c_j^{n+1}\Capa(K_j), \quad j=1,2.$$
For any fixed $t$, the function $u_t^c=-c_t$ on $K_t^c$, so $u_t^c\le -c_t\,\omega_{K_t^c}$. By \cite[Cor. 2.2]{R17a}, this implies
$$ \E(u_t^c)\le \E(c_t\,\omega_{K_t^c})=-c_t^{n+1}\Capa(K_t^c),$$
and  (\ref{wlincap}) follows from the geodesic linearization property (\ref{lin}).

{\it (iii)} It suffices to prove the concavity of the function $V$.
When the weights $c_j$ satisfy (\ref{equil}), inequality  (\ref{wlincap}) rewrites as
$$V(t)\ge  \frac{c_t}{c_0}V(0)$$
and, since
$$ c_1= \frac{V(1)}{V(0)}c_0,$$
this gives us
$$V(t)\ge (1-t)\,V(0)+t\,V(1),$$
which completes the proof.
\end{proof}

\medskip
The convexity/concavity results in this theorem are stronger than inequality (\ref{lincap}) obtained in \cite{R17a} by the geodesic interpolation $u_t$ of non-weighted extremal functions. In addition, the non-weighted geodesic $u_t$ is very unlikely to be the extremal function of the set $K_t$ (as shown in \cite{R17b}, this is never the case in the toric situation, unless $K_0=K_1$), while this is perfectly possible in the weighted interpolation. For example, given $K_0\Subset \Omega$, let
$$ K_1=\left\{z\in\Omega:\: \omega_{K_0}(z)\le-\frac12\right\},$$
then $\omega_{K_1,\Omega}=\max\{2\omega_{K_0,\Omega},-1\}$. For $c_0=2$ and $c_1=1$ we get
$$u_t^c=\max\{2\omega_{K_0},-2+t\}= (2-t)\,\omega_{K_t^c}$$
with
$$K_t^c=\{z\in\Omega:\: \omega_{K_0}(z)\le-1+t/2\},$$
so $$\Capa(K_t^c)=\left(1-\frac{t}2\right)^{-1}|\E(u_t^c)|=\left(1-\frac{t}2\right)^{-1-n}\Capa(K_0).$$


\section{Toric case}

In this section, we assume $\Omega=\D^n$, the unit polydisk, and $K_0,K_1\subset\D^n$ to be non-pluripolar, polynomially convex compact Reinhardt (multicircled, toric) sets. Polynomial convexity of such a set $K$ means that its logarithmic image
$$\Log K=\{s\in\Rnm:\: (e^{s_1},\ldots,e^{s_n})\in K\}$$
is a {\it complete} convex subset of $\Rnm$, i.e., $\Log K+\Rnm\subset\log K$; we will also say that $K$ is {\it complete logarithmically convex}. The functions $c_j\,\omega_{K_j}$ are toric, and so is their geodesic $u_t^c$. Note that since $K_0\cap K_1\neq\emptyset$, inequality (\ref{wlincap}) and the concavity/convexity statements of Theorem~\ref{gentheo}(iii) hold true.

It was shown in \cite{CER} that the sets $K_t$ defined by (\ref{K_t}) for the geodesic interpolation of non-weighted toric extremal functions $\omega_{K_j}$ are, as in the classical interpolation theory, the geometric means $K_t^\times$  of $K_j$.
 Here we extend the result to the weighted interpolation, which shows, in particular, that the sets $K_t^c$ do not depend on the weights $c_j$.
The relation $K_t^\times\subset K_t^c$ is easy, while the reverse inclusion is more elaborate; we mimic the proof of the corresponding relation for the non-weighted case \cite{CER} that rests on a machinery developed in \cite{R17b}.

Any toric plurisubharmonic function $u(z)$ in $\D^n$ gives rise to a convex function
\beq\label{convim}\check u(s)=u (e^{s_1},\ldots,e^{s_n}), \quad s\in\Rnm,\eeq
and the geodesic $u_t^c$ generates the function $\check u_t^c$, convex in $(s,t)\in \Rnm\times (0,1)$.

Given a convex function $f$ on $\Rnm$, we extend it to the whole $\Rn$ as a lower semicontinuous convex function on $\Rn$, equal to $+\infty$ on $\Rn\setminus{\overline \Rnm}$, and we denote $\cL[f]$ its {\it Legendre transform}:
$$ \cL[f](x)=\sup_{y\in\Rn}\{\langle x,y\rangle -f(y)\}.$$
Evidently, $\cL[f](x)=+\infty$ if $x\not\in\overline{\Rnp}$, and the Legendre transform is an involutive duality between convex functions on $\Rnp$ and $\Rnm$.

It was shown in \cite{R17b} that for the relative extremal function $\omega_K=\omega_{K,\D^n}$,
$$\cL[\check \omega_K]=\max\{h_Q+1,0\},$$
where
$$h_Q(a)=\sup_{s\in Q} \langle a,s\rangle,\quad a\in\Rnp
$$
is the support function of the convex set $Q=\Log K \subset\Rnm$.
Therefore, for a weighted relative extremal function $u=c\,\omega_{K}$ we have
\beq\label{Lu} \cL[\check u](a) = c_j\,\cL[\check \omega_{K_j}](c_j^{-1}a)
= \max\{h_Q(a)+c_j,0\}.
\eeq

\begin{theorem}\label{cgeom} Given two non-pluripolar complete logarithmically convex compact Reinhardt sets  $K_0,K_1\subset\D^n$ and two positive numbers $c_0$ and $c_1$, let $u_t^c$ be the geodesic connecting the functions $u_0=c_0\,\omega_{K_0}$ and $u_1=c_1\,\omega_{K_1}$. Then the interpolating sets $K_t^c$ defined by (\ref{K_t^c}) coincide with the geometric means
$$K_t^\times:=K_0^{1-t}K_1^t=\{z:\: |z_l|=|\eta_l|^{1-t} |\xi_l|^{t}, \ 1\le l\le n,\ \eta\in K_0,\ \xi\in K_1\}.$$
\end{theorem}

\begin{proof} Since the sets $K_t^\times$ and $K_t^c$ are complete logarithmically convex, it suffices to prove that $Q_t:=\Log K_t^\times $ coincides with $Q_t^c:=\Log K_t^c$.

The inclusion $Q_t\subset Q_t^c$ follows from convexity of the function $\check u_t^c(s)$ in $(s,t)\in \Rnm\times (0,1)$: if $s\in Q_t$, then $s=(1-t)\,s_0+t\,s_1$ for some $s_j\in Q_j$, so
$$ \check u_t(s)\le (1-t)\, \check u_0(s_0)+ t\, \check u_1(s_1)=c_t,$$
while we have $\check u_t(s)\ge -c_t$ for all $s$. This gives us $s\in Q_t^c$.

To prove the reverse inclusion, take an arbitrary point $\xi\in\Rnm\setminus Q_t$, then there exists $b\in\Rnp$ such that
\beq\label{out} \langle b,\xi\rangle > h_{Q_t}(b)=(1-t)h_{Q_0}(b)+ t\, h_{Q_1}(b).\eeq
By the homogeneity of the both sides, we can assume $h_{Q_0}(b)\ge-c_0$ and $h_{Q_1}(b)\ge-c_1$. Then, by (\ref{Lu}) and (\ref{out}), we have
\begin{eqnarray*}
\check u_t(\xi) &=& \sup_{a}[\langle a,\xi\rangle - (1-t)\max\{h_{Q_0}(a)+c_0,0\} - t \max\{h_{Q_1}(a)+c_1,0\}]\\
&\ge & \langle b,\xi\rangle - (1-t)\max\{h_{Q_0}(b)+c_0,0\} - t \max\{h_{Q_1}(b)+c_1,0\}\\
&>& (1-t)[h_{Q_0}(b)-(h_{Q_0}(b)+1)] + t [h_{Q_1}(b)-(h_{Q_1}(b)+1)]=-1,
\end{eqnarray*}
so $\xi\not\in Q_t^c$.
\end{proof}

\medskip
Now the corresponding assertions of Theorem~\ref{gentheo} can be stated as follows.

\begin{theorem}\label{torcap} For non-pluripolar complete logarithmically convex compact Reinhardt sets  $K_0,K_1\subset\D^n$, the inequality
\beq\label{torBM} c_t^{n+1}\Capa(K_t^\times, \D^n)\le (1-t)\,c_0^{n+1}\Capa(K_0,\D^n)+t\,c_1^{n+1}\Capa(K_1,\D^n)\eeq
holds true for any $c_0,c_1>0$ and $c_t=(1-t)\,c_0+t\,c_1$.

In particular, the function
$$t\mapsto \left(\Capa(K_t^\times,\D^n)\right)^{-\frac1{n+1}}$$
is concave and, consequently, the function
$$t\mapsto\left(\Capa(K_t^\times,\D^n)\right)^{\frac1{n+1}}$$
is convex.
\end{theorem}

As we saw in the example in the previous section, sometimes one can have $u_t=\omega_{K_t^c}$ for $u_j=c_j\,\omega_{K_j}$, in which case (\ref{torBM}) becomes an equality. Our next result determines when this is possible for the toric case.

\begin{theorem}\label{equal} In the conditions of Theorem~\ref{cgeom}, the geodesic $u_\tau^c$ equals $c_\tau\,\omega_{K_\tau}$ for some $\tau\in (0,1)$ if and only if
$$K_1^{c_0}=K_0^{c_1},$$ that is, $c_0\,\Log K_1 =c_1\,\Log K_0$.
\end{theorem}

\begin{proof} We will use the toric geodesic representation formula established in \cite[Thm. 5.1]{R17b},
\beq\label{guan}\check u_t =\cL\left[ (1-t)\cL[\check u_0] + t\cL[\check u_1]\right],\eeq
which is a local counterpart of Guan's result \cite{Gu} for compact toric manifolds;
here $\check u$ is the convex image (\ref{convim}) of the toric plurisubharmonic function $u$.

Let $Q_t=\log K_t$, $0\le t\le 1$.
By (\ref{Lu}), $u_\tau^c=c_\tau\,\omega_{K_\tau}$ means
$$ (1-\tau)\max\{h_{Q_0}(a)+c_0,0\} + \tau \max\{h_{Q_1}(a)+c_1,0\}=\max\{h_{Q_\tau}(a)+c_\tau,0\},$$
or, which is the same,
$$ \max\{h_{(1-\tau)Q_0}(a)+(1-\tau)c_0,0\} +  \max\{h_{\tau Q_1}(a)+\tau c_1,0\}=\max\{h_{Q_\tau}(a)+c_\tau,0\}$$
for all $a\in\Rnp$. Therefore, $h_{Q_0}(a)\le -c_0$ if and only if $h_{Q_1}(a)\le -c_1$, so $c_0\,Q_0^\circ = c_1\,Q_1^\circ$ and, since $(c \,Q)^\circ =c^{-1}Q^\circ$, we get
$c_0\,Q_1=c_1\,Q_0$. Here $Q^\circ$ is the {\it copolar} (\ref{copolar}) to the set $Q$, see the beginning of the next section.
\end{proof}

\section{Covolumes}

In the toric case, the Monge-Amp\`ere capacities with respect to the unit polydisk can be represented as volumes of certain sets \cite{ARZ}, \cite{R17b}. Namely, if $K\Subset\D^n$ is complete and logarithmically convex, then $Q:=\Log K\subset \Rnm$ and
\beq\label{capcovol}\Capa (K,\D^n)=n!\,\Covol(Q^\circ):=n!\, \Vol({\Rnp\setminus Q^\circ}),\eeq
where the convex set $Q^\circ\subset\Rnp$ defined by
\beq\label{copolar}Q^\circ=\{x\in\Rn: h_Q(x) \le -1 \}= \{x\in\Rn: \langle x,y\rangle \le -1 \ \forall y\in Q\}\eeq
is, in the terminology of \cite{R17b}, the {\it copolar} to the set $Q$. In particular,
$$\Capa(K_t^\times,\D^n)=n!\,\Covol(Q_t^\circ)$$
for the copolar $Q_t^\circ$ of the set $Q_t=(1-t)Q_0+t\,Q_1$; we would like to stress that $Q_t^\circ\neq (1-t)Q_0^\circ+t\,Q_1^\circ$.

Convex complete subsets $P$ of $\Rnp$ (i.e., $P+\Rnp\subset P$) appear in singularity theory and complex analysis (see, for example,  \cite{KiR}, \cite{KaK}, \cite{KT}, \cite{Ku1}, \cite{R00}, \cite{R09}), their {\it covolumes} (the volumes of $\Rnp\setminus P$) being used for computation of the multiplicities of mappings, etc.  Such a set $P$ generates, by the same formula (\ref{copolar}), its copolar $P^\circ\subset\Rnm$, whose exponential image $\Exp P^\circ$ (the closure of all points $(e^{s_1},\ldots,e^{s_n})$ with $s\in P^\circ$) is a complete logarithmically convex subset of $\D^n$. Since taking the copolar is an involution, the representation (\ref{capcovol}) translates coherently the inequalities on the capacities to those on the (co)volumes. Namely, $\Capa(Q_j)$ becomes replaced by $\Covol(P_j)$ with $P_j=Q_j^\circ\subset\Rnp$ for $j=0,1$, while $\Capa(Q_t)$ has to be replaced with the covolume of the set whose copolar is $Q_t$, that is, with $\left((1-t)\,P_0^\circ+t\,P_1^\circ\right)^\circ$. The operation of {\it copolar addition}
$$P_0\oplus P_1:=\left(P_0^\circ+P_1^\circ\right)^\circ$$
was introduced in \cite{R17b}. In particular, it was shown there that the copolar sum of any pair of cosimplices in $\Rnp$, unlike their Minkowski sum, is still a simplex.

\medskip

\begin{corollary}\label{torcovol} Let  $P_0,P_1$ be non-empty convex complete subsets of $\Rnp$, and let the interpolating sets $P_t^\oplus$ be defined by
$$ P_t^\oplus=\left((1-t)P_0^\circ+tP_1^\circ\right)^\circ,\quad 0< t< 1.$$ Then the inequality
$$ c_t^{n+1}\Covol(P_t^\oplus)\le (1-t)\,c_0^{n+1}\Covol(P_0)+t\,c_1^{n+1}\Covol(P_1)$$
holds true for any $c_0,c_1>0$ and $c_t=(1-t)\,c_0+t\,c_1$.

In particular, the function
$$ V^\oplus[P](t):=\left(\Covol(P_t^\oplus)\right)^{-\frac1{n+1}}$$
is concave and, consequently, the function
$$ \rho^\oplus[P](t):=\left(\Covol(P_t^\oplus)\right)^{\frac1{n+1}}$$
is convex.
\end{corollary}

\medskip

Note that the convexity of $\rho^\oplus$ (following from the concavity of $V^\oplus$) implies that the function $$\tilde\rho^\oplus[P](t):= \left(\Covol(P_t^\oplus)\right)^{\frac1{n}}$$ is convex as well. Since $\tilde\rho^\oplus$ is a homogeneous function of $P$, that is, $$\tilde\rho^\oplus[c\,P](t)=c\,\tilde\rho^\oplus[P](t)$$ for all $c>0$ and $0<t<1$, its convexity is equivalent to the logarithmic convexity of the covolumes, established in \cite{CER} by convex geometry methods:
$$ \Covol(P_t^\oplus)\le \Covol(P_0)^{1-t}\Covol(P_1)^t, $$
which is just another form of the Brunn-Minkowski type inequality (\ref{logcap}).
Therefore, the concavity of the function $V^\oplus$ is a stronger property then just the logarithmic convexity of the covolumes.

\bigskip
{\bf Acknowledgemengt.} The author thanks the anonymous referee for suggestions that have improved the presentation.

\bigskip\noindent
{\sc Tek/Nat, University of Stavanger, 4036 Stavanger, Norway}

\noindent
\emph{e-mail:} alexander.rashkovskii@uis.no


\begin{thebibliography}{11}

\bibitem{A}
S. Abja, \textit{Geometry and topology of the space of plurisubharmonic functions},
J. Geom. Anal. {\bf 29} (2019), no. 1, 510--541.

\bibitem{ARZ}
A. Aytuna, A. Rashkovskii and V. Zahariuta,  \textit{ Widths asymptotics for a pair of
Reinhardt domains}, Ann. Polon. Math. {\bf 78} (2002), 31--38.

\bibitem{BL}
J. Bergh, J. L\"ofstr\"om, Interpolation Spaces. An Introduction. Springer, 1976.


\bibitem{BB}
R.J. Berman and B. Berndtsson,  \textit{Moser-Trudinger type inequalities for complex Monge-Ampère operators and Aubin's "hypoth\`{e}se fondamentale"},  arXiv:1109.1263.


\bibitem{Ca}
A.-P. Calder\'{o}n,  \textit{Intermediate spaces and interpolation, the complex method},
Studia Math. {\bf 24} (1964), 113--190.


\bibitem{CE}
D. Cordero-Erausquin,  \textit{Santal\'{o}'s inequality on $\Cn$ by complex interpolation}, C. R. Acad. Sci. Paris, Ser. I {\bf 334} (2002), 767--772.

\bibitem{CEK}
D. Cordero-Erausquin and B. Klartag,  \textit{Interpolations, convexity and geometric inequalities}, in: Geometric aspects of functional analysis, 151--168, Lecture Notes in Math., 2050, Springer, 2012.

\bibitem{CER}
D. Cordero-Erausquin and A. Rashkovskii, \textit{ Plurisubharmonic geodesics and interpolating sets}, to appear in Arch. Math.; arXiv:1807.09521.

\bibitem{Gu}
{D. Guan},  \textit{On modified Mabuchi functional and Mabuchi moduli space of K\"{a}hler metrics on toric bundles},
Math. Res. Lett. {\bf 6} (1999), no. 5-6, 547--555.

\bibitem{G12}
V. Guedj (ed.), Complex Monge-Ampère equations and geodesics in the space of Kähler metrics. Lecture Notes in Math., {\bf 2038}, Springer, 2012.

\bibitem{KiR}
D. Kim and A. Rashkovskii,  \textit{ Higher Lelong numbers and convex geometry}, arXiv:1803.07948; to appear in J. Geom. Anal.




 \bibitem{KaK}
 K. Kaveh and A. Khovanskii,   \textit{ Convex bodies and multiplicities of ideals}, Proc. Steklov Inst. Math. {\bf 286} (2014), no. 1, 268--284.


\bibitem{KT}
A. Khovanskiĭ and V. Timorin,  \textit{On the theory of coconvex bodies},
Discrete Comput. Geom. {\bf 52} (2014), no. 4, 806--823.

\bibitem{Kl}
{ M. Klimek}, Pluripotential theory. Oxford University Press,
London, 1991.

\bibitem{Ku1}
{ A.G. Kouchnirenko},  \textit{ Poly\`edres de Newton et
nombres de Milnor}, Invent. Math. {\bf 32} (1976), 1-31.

\bibitem{R00}
{ A. Rashkovskii},  \textit{ Newton numbers and residual measures of
  plurisubharmonic functions}, Ann. Polon. Math. {\bf 75} (2000),
no. 3, 213--231.

\bibitem{R09}
{ A. Rashkovskii}, \textit{ Tropical analysis of plurisubharmonic singularities}, in: Tropical and Idempotent Mathematics, 305--315, Contemp. Math., {\bf 495}, Amer. Math. Soc., Providence, RI, 2009.

\bibitem{R17a}
{ A. Rashkovskii}, \textit{ Local geodesics for plurisubharmonic functions}, Math. Z. {\bf 287} (2017), 73--83.

\bibitem{R17b}
{ A. Rashkovskii}, \textit{Copolar convexity}, Ann. Polon. Math. {\bf 120} (2017), no. 1, 83--95.



\end{thebibliography}
\end{document}